\documentclass[11pt,times,twoside]{article}

\usepackage[leqno]{amsmath}
\usepackage{graphicx}
\usepackage{latexsym}
\usepackage{amsmath,amsfonts,amssymb,amsthm,euscript,makeidx,color,mathrsfs}
\usepackage{enumerate}

\usepackage[colorlinks,linkcolor=blue,anchorcolor=green,citecolor=red]{hyperref}

\def\sqr#1#2{{\vcenter{\vbox{\hrule height.#2pt
              \hbox{\vrule width.#2pt height#1pt \kern#1pt \vrule width.#2pt}
              \hrule height.#2pt}}}}
\def\5n{\negthinspace \negthinspace \negthinspace \negthinspace \negthinspace }
\def\4n{\negthinspace \negthinspace \negthinspace \negthinspace }
\def\3n{\negthinspace \negthinspace \negthinspace }
\def\2n{\negthinspace \negthinspace }
\def\1n{\negthinspace }

\def\dbE{\mathbb{E}}

\def\dbP{\mathbb{P}}

\def\dbR{\mathbb{R}}

\def\ds{\displaystyle}

\def\ns{\noalign{\ss}}
\def\a{\alpha}
\def\b{\beta}
\def\g{\gamma}

\def\o{\omega}
\def\p{\phi}
\def\i{\infty}
\def\Om{\Omega}
\def\cF{{\cal F}}

\def\cV{{\cal V}}

\def\ba{\bar{a}}

\def\bc{\bar{c}}

\def\be{\bar{e}}

\def\bl{\bar{l}}

\def\bp{\bar{p}}

\def\br{\bar{r}}

\def\bt{\bar{t}}

\def\ss{\smallskip}
\def\ms{\medskip}

\def\qq{\qquad}

\def\ra{\mathop{\rightarrow}}

\def\rf{\eqref}

\def\({\Big (}
\def\){\Big )}
\def\[{\Big[}
\def\]{\Big]}

\def\bde{\begin{definition}\label}
\def\ede{\end{definition}}
\def\be{\begin{equation}}
\def\bel{\begin{equation}\label}
\def\ee{\end{equation}}
\def\bt{\begin{theorem}\label}
\def\et{\end{theorem}}
\def\bc{\begin{corollary}\label}
\def\ec{\end{corollary}}
\def\bl{\begin{lemma}\label}
\def\el{\end{lemma}}
\def\bp{\begin{proposition}\label}
\def\ep{\end{proposition}}
\def\bas{\begin{assumption}\label}
\def\eas{\end{assumption}}
\def\br{\begin{remark}\label}
\def\er{\end{remark}}
\def\bex{\begin{example}\label}
\def\ex{\end{example}}
\def\ba{\begin{array}}
\def\ea{\end{array}}
\def\ed{\end{document}}
\def\ts{\times}
\def\olb\
\def\eps{\epsilon}
\def\square#1{\vbox{\hrule\hbox{\vrule height#1%
     \kern#1\vrule}\hrule}}
\def\rectangle#1#2{\vbox{\hrule\hbox{\vrule height#1%
     \kern#2\vrule}\hrule}}

\font\tenbb=msbm10 \font\sevenbb=msbm7 \font\fivebb=msbm5

\newfam\bbfam
\scriptscriptfont\bbfam=\fivebb \textfont\bbfam=\tenbb
\scriptfont\bbfam=\sevenbb

\newtheorem{theorem}{\hskip 1.3em Theorem}[section]
\newtheorem{definition}[theorem]{\hskip 1.3em Definition}
\newtheorem{proposition}[theorem]{\hskip 1.3em Proposition}
\newtheorem{corollary}[theorem]{\hskip 1.3em Corollary}
\newtheorem{lemma}[theorem]{\hskip 1.3em Lemma}
\newtheorem{remark}[theorem]{\hskip 1.3em Remark}
\newtheorem{example}[theorem]{\hskip 1.3em Example}

\newtheorem{assumption}[theorem]{\hskip 1.3em Assumption}

\makeatletter
   
   \@addtoreset{equation}{section}
\makeatother

\begin{document}

\title{\Large \bf Symmetrical martingale solutions of backward doubly stochastic Volterra integral equations\footnotemark[1]}

\author
{\textbf{Jiaqiang Wen}$^{1,2}$,~~~~ \textbf{Yufeng Shi}$^{1,\dagger}$ \\
\normalsize{$^{1}$Institute for Financial Studies and School of Mathematics,}\\
\normalsize{Shandong University, Jinan 250100, Shandong, China}\\
\normalsize{$^{2}$Department of Mathematics, Southern University of Science and Technology,} \\
\normalsize{Shenzhen 518055, Guangdong, China}\\
}

\date{}
\renewcommand{\thefootnote}{\fnsymbol{footnote}}

\footnotetext[1]{This work is supported by National Natural Science Foundation of China (Grant Nos. 11871309, 11371226, 11671229, 11071145, 11526205, 11626247 and 11231005), the Foundation for Innovative Research Groups of National Natural Science Foundation of China (Grant No. 11221061) and the 111 Project (Grant No. B12023), the National Key R\&D Program of China (Grant No. 2018YFA0703900)}

\footnotetext[2]{Corresponding author. E-mail address: jqwen59@gmail.com (J. Wen), yfshi@sdu.edu.cn (Y. Shi).}

\maketitle

\begin{abstract}
This paper aims to study a new class of integral equations called backward doubly stochastic Volterra integral equations (BDSVIEs, for short).
The notion of symmetrical martingale solutions (SM-solutions, for short) is introduced for BDSVIEs.
And the existence and uniqueness theorem for BDSVIEs in the sense of SM-solutions is established.
\end{abstract}

\textbf{Keywords}:
Backward doubly stochastic Volterra integral equation; Backward stochastic integral; Symmetrical martingale solution.

\vspace{3mm}

\textbf{2010 Mathematics Subject Classification}: 60H10; 60H20

\section{Introduction}

Throughout this paper, let $(\Omega,\mathcal{F},P)$ be a probability space and $T>0$ be a fixed terminal time.
Let $\{W_{t};0\leq t\leq T\}$ and $\{B_{t};0\leq t\leq T\}$ be two mutually independent standard Brownian motion processes,
 with values respectively in $\mathbb{R}^{d}$ and in $\mathbb{R}^{l}$, defined on $(\Omega,\mathcal{F},P)$.
Let $\mathcal{N}$ denote the class of $P$-null sets of $\mathcal{F}$.
For each $t\in [0,T]$, we define
 \begin{equation*}
 \mathcal{F}_{t}\triangleq\mathcal{F}_{t}^{W} \vee \mathcal{F}_{t,T}^{B},
 \end{equation*}
 where for any process $\{\eta_{t}\}, \ \mathcal{F}_{s,t}^{\eta}=\sigma\{\eta_{r}-\eta_{s};s\leq r\leq t\}\vee \mathcal{N}$ and
 $\mathcal{F}_{t}^{\eta}=\mathcal{F}_{0,t}^{\eta}$.
Denote $\Delta=\{ (t,s)\in[0,T]^{2}| \ t\leq s\}$ and $\Delta^{c}=\{ (t,s)\in[0,T]^{2}| \ s < t\}$.
Consider the following integral equation:
\begin{equation}\label{46}
 Y(t)=\psi(t) + \int_t^T f(t,s,Y(s),Z(t,s),Z(s,t)) ds  - \int_t^T Z(t,s) dW_{s}, \ \ t\in [0,T],
\end{equation}
where $f(\omega,t,s,y,z,z'):\Omega\times \Delta \times \mathbb{R}^{k}\times \mathbb{R}^{k\times d}\times
\mathbb{R}^{k\times d}\rightarrow \mathbb{R}^{k}$
 is an $\mathcal{F}^{W}_{s}$-adapted process
and $\psi(\omega,t):\Omega\times [0,T]\rightarrow \mathbb{R}^{k}$ is an $\mathcal{F}^{W}_{T}$-measurable process.
Such an equation is referred to as a backward stochastic Volterra integral equation (BSVIE, for short) introduced by Yong in \cite{Yong2, Yong4}.
A special case of (\ref{46}) with $f(\cdot)$ independent of $Z(s,t)$ and $\psi(t)\equiv \xi$ was studied in \cite{Lin} a little earlier.
Since then,
there are many applications of BSVIEs among stochastic optimal control problems \cite{Shi4},
risk management \cite{Yong3} and capital allocations \cite{Eduard}.
Some other recent developments of BSVIEs can be found in
\cite{Agram},
\cite{Anh},
\cite{Djordjevic},
\cite{Ren},
\cite{Shi2},
\cite{Wang3},
\cite{Yong},
\cite{Wangz},
\cite{Yong5},
\cite{Zhangx},
and so on.
As interpreted in \cite{Yong4}, in order to guarantee the uniqueness of solutions of BSVIE (\ref{46}),
some additional constraints should be imposed on $Z(t,s)$ as $(t,s)\in \Delta^{c}$.
Then Yong \cite{Yong4} introduced the martingale solutions (M-solutions, for short) for BSVIE (\ref{46}).
Also the symmetrical solutions (S-solutions, for short) of (\ref{46}) was introduced in \cite{Wang2}.

On the other hand, Pardoux and Peng introduced so-called backward doubly stochastic differential equations (BDSDEs, for short) in \cite{Peng} as follows:
\begin{equation}\label{BDSDE}
\begin{split}
  Y(t)=&\ \xi + \int_t^T f(s,Y(s),Z(s)) ds\\
       &+ \int_t^T g(s,Y(s),Z(s)) d\overleftarrow{B}_{s}- \int_t^T Z(s) dW_{s}, \qq t\in [0,T].
\end{split}
\end{equation}
Since then, there are many literatures on the theory of BDSDEs.
For example, \cite{Han}, \cite{Shi3}, \cite{Shi}, \cite{Shi5, Shi6}, \cite{Shi15}, etc.,
have developed the theory and applications of BDSDEs.

In this paper, motivated by the above works, we study the following stochastic integral equation:
\begin{equation}\label{1}
\begin{split}
  Y(t)=&\ \psi(t) + \int_t^T f(t,s,Y(s),Z(t,s),Z(s,t)) ds\\
       &+ \int_t^T g(t,s,Y(s),Z(t,s),Z(s,t)) d\overleftarrow{B}_{s}- \int_t^T Z(t,s) dW_{s}, \qq t\in [0,T],
\end{split}
\end{equation}
where
$f(\omega,t,s,y,z,z'):\Omega\times \Delta \times \mathbb{R}^{k}\times \mathbb{R}^{k\times d}\times \mathbb{R}^{k\times d}\rightarrow \mathbb{R}^{k}$ and
$g(\omega,t,s,y,z,z'):\Omega\times \Delta \times \mathbb{R}^{k}\times \mathbb{R}^{k\times d}\times \mathbb{R}^{k\times d}
\rightarrow \mathbb{R}^{k\times d}$
are $\mathcal{F}_{s}$-measurable,
and $\psi(\omega,t):\Omega\times [0,T]\rightarrow \mathbb{R}^{k}$ is $\mathcal{F}_{T}$-measurable.
We call this new type of equations (\ref{1}) as backward doubly stochastic Volterra integral equations (BDSVIEs, for short).
Obviously BDSVIE (\ref{1}) is a generalization of both BSVIE (\ref{46}) and BDSDE (\ref{BDSDE}).

Comparing with BSVIE (\ref{46}), we notice that there are two independent Brownian motions $W(t)$ and $B(t)$ in BDSVIE (\ref{1}),
 where the $dW$-integral is a forward It\^{o}'s integral and the $d\overleftarrow{B}$-integral is a backward It\^{o}'s integral.
 The extra noise $B$ in the equation could describe some extra information that cannot be detected in practice, such as insider information in a financial market,
 which is available only for some investors.
However, the extra term $d\overleftarrow{B}$ would bring some extra difficulties.
 We will overcome the difficulties by finding some suitable assumptions on the coefficient $g$.
 For BDSVIE (\ref{1}), similar to BSVIE (\ref{46}), in order to guarantee the uniqueness of solutions,
 some restrictions also should be imposed on $Z(t,s)$ as $(t,s)\in \Delta^{c}$.
However, $Y(\cdot)$ and $Z(t,\cdot)$ are $\mathcal{F}$-measurable, not $\mathcal{F}^{W}$-adapted, so the measurability of solutions for BDSVIEs is extremely complicated.
Therefore we have to introduce a new definition of measurable solutions for BDSVIE (\ref{1}).
In this paper we firstly introduce the notion of symmetrical martingale solutions (SM-solutions, for short) for BDSVIEs.
It is worth noting that the SM-solutions are different from both the M-solutions (refer to \cite{Yong4}) and the S-solutions (refer to \cite{Wang2}) in the theory of BSVIEs.
Then we can luckily establish the existence and uniqueness theorem for BDSVIEs in the sense of SM-solutions.
It is worth to point out that BDSVIE (\ref{1}) could be applied in many fields such as mathematical finance, risk management and stochastic optimal controls and so on.
And we expect to study more applications of BDSVIEs and more properties of the SM-solutions in the future works.
The connection between BDSVIEs and stochastic PDEs is another important issue in our future investigations.

The remaining of this paper is organized as follows.
In Section 2, some notations and preliminary results are presented.
We introduce the SM-solutions for BDSVIEs in Section 3. And the existence and uniqueness of BDSVIEs in the sense of SM-solutions is proved in this section.
For the conciseness of the paper, in Section 4, as an appendix we show a detail proof for the backward martingale representation theorem which will play an important role in this paper.

\section{Preliminaries}

$\mathbf{Notation.}$ The Euclidean norm of a vector $x\in\mathbb{R}^{k}$ will be denoted by $|x|$,
and for a $k\times d$ matrix $A$, we define $\| A \|=\sqrt{Tr AA^{\ast}}$.
For simplicity, let $d=l$ and for $H=\mathbb{R}^{k}, \ \mathbb{R}^{k\times d}$, $t\in[0,T]$, denote
\begin{itemize}
  \item [$\bullet$] $L^{2}(\mathcal{F}_{T};H)
        =\{\xi:\Omega\rightarrow H \mid \xi$ is $\mathcal{F}_{T}$-measurable, $E[|\xi|^{2}]< \infty\}$;
  \item [$\bullet$] $L_{\mathcal{F}_{T}}^2(0,T;H)
 =\big \{\psi:\Omega\times [0,T]\rightarrow H \mid \psi(t)$ is  $\mathcal{F}_{T}$-measurable,
        $E\int_0^T |\psi(t)|^{2} dt < \infty\big \}$;
  \item [$\bullet$] $L_{\mathcal{F}}^2(0,T;H)
        =\big \{X:\Omega\times [0,T]\rightarrow H \mid X(t)$ is  $\mathcal{F}_{t}$-measurable,
        $E\int_0^T |X(t)|^{2} dt$ $< \infty\big \}$;
   \item [$\bullet$] $S_{\mathcal{F}}^2(0,T;H)
  =\big\{X:\Omega\times [0,T]\rightarrow H \mid X(t)$ is an $\mathcal{F}_{t}$-measurable process, has continuous paths,
  $E\big(\sup\limits_{0\leq t\leq T}|X(t)|^{2}\big)< \infty\big\}$;
  \item [$\bullet$] $L_{\mathcal{F}}^2(\Delta;H)
        =\big\{Z:\Omega\times \Delta \rightarrow H  \mid Z(t,s)$ is $\mathcal{F}_{s}$-measurable,
        $E\int_0^{T} \int_t^{T} |Z(t,s)|^{2} dsdt$ $< \infty\big\}$;
  \item [$\bullet$] $L_{\mathcal{F}}^2([0,T]^{2};H)
        =\big\{Z:\Omega\times [0,T]\times [0,T] \rightarrow H  \mid Z(t,s)$ is an $\mathcal{F}_{s}$-measurable process,
        $E\int_0^{T} \int_0^{T} |Z(t,s)|^{2} dsdt$ $< \infty\big\}$.
\end{itemize}
Similarly, we can define $L^{2}_{\mathcal{F}}(\Delta^{c};H), L^{2}_{\mathcal{F}^{W}}(\Delta^{c};H), L^{2}_{\mathcal{F}^{B}}(\Delta;\mathbb{R}^{k\times d})$, etc.
Also, let
\begin{equation*}
\mathcal{H}_{\Delta}^{2}[0,T]=L_{\mathcal{F}}^2(0,T;\mathbb{R}^{k})\times L_{\mathcal{F}}^2(\Delta;\mathbb{R}^{k\times d}); \ \
\mathcal{H}^{2}[0,T]=L_{\mathcal{F}}^2(0,T;\mathbb{R}^{k})\times L_{\mathcal{F}}^2([0,T]^{2};\mathbb{R}^{k\times d}).
\end{equation*}
It's easy to see that $\mathcal{H}_{\Delta}^{2}[0,T]$ and $\mathcal{H}^{2}[0,T]$ are Hilbert spaces.

 \begin{itemize}
   \item [(H1)]
   Let
\begin{equation*}
\begin{split}
      &f:\Omega\times [0,T]\times \mathbb{R}^{k}\times \mathbb{R}^{k\times d}\rightarrow \mathbb{R}^{k};\\
      &g:\Omega\times [0,T]\times \mathbb{R}^{k}\times \mathbb{R}^{k\times d}\rightarrow \mathbb{R}^{k\times d},
\end{split}
\end{equation*}
be jointly measurable such that for any $(y,z)\in \mathbb{R}^{k}\times \mathbb{R}^{k\times d}$,
\begin{equation*}
  f(\cdot,y,z)\in L_{\mathcal{F}}^{2}(0,T; \mathbb{R}^{k}), \ \ g(\cdot,y,z)\in L_{\mathcal{F}}^{2}(0,T; \mathbb{R}^{k\times d}).
\end{equation*}
Moreover, we assume that there exist constants $c>0$ and $0<\alpha<1$ such that
 for any $(\omega,t)\in \Omega\times [0,T],(y,z),(y',z')\in \mathbb{R}^{k}\times \mathbb{R}^{k\times d}$,
\begin{equation*}
 \begin{cases}
   |f(t,y,z)-f(t,y',z')|^{2}\leq c(|y-y'|^{2} + \|z-z'\|^{2});\\
   \|g(t,y,z)-g(t,y',z')\|^{2}\leq c|y-y'|^{2} + \alpha \|z-z'\|^{2}.\\
 \end{cases}
\end{equation*}
 \end{itemize}
The following proposition is from Pardoux and Peng \cite{Peng}.

\begin{proposition}\label{3}
Let (H1) hold, then for any $\xi\in L^2(\mathcal{F}_{T};\mathbb{R}^{k})$, the following BDSDE
\begin{equation}\label{48}
   Y(t)=\xi + \int_t^T f(s,Y(s),Z(s)) ds + \int_t^T g(s,Y(s),Z(s)) d\overleftarrow{B}_{s} - \int_t^T Z(s) dW_{s}, \qq t\in[0,T],
\end{equation}
admits a unique solution
$(Y,Z)\in S_{\mathcal{F}}^{2}(0,T; \mathbb{R}^{k})\times L_{\mathcal{F}}^{2}(0,T; \mathbb{R}^{k\times d}).$
\end{proposition}

\section{Main results}

In this section, we study the existence and uniqueness result for BDSVIEs. Firstly, we introduce the symmetrical martingale solution.

\subsection{Symmetrical martingale solution}

Consider the general type of BDSVIEs as follows:
\bel{55}\ba{ll}
\ds Y(t)= \psi(t) + \int_t^T f(t,s,Y(s),Z(t,s),Z(s,t)) ds\\
\ns\ds\qq \ \ \ \ + \int_t^T g(t,s,Y(s),Z(t,s),Z(s,t))d\overleftarrow{B}_{s}-\int_t^T Z(t,s) dW_{s}, \qq t\in [0,T],
\ea\ee
where
$f(\omega,t,s,y,z,\zeta):\Omega\times \Delta\times \mathbb{R}^{k}\times \mathbb{R}^{k\times d}\times \mathbb{R}^{k\times d}\rightarrow \mathbb{R}^{k}$
and
$g(\omega,t,s,y,z,\zeta):\Omega\times \Delta\times \mathbb{R}^{k}\times \mathbb{R}^{k\times d}\times \mathbb{R}^{k\times d}\rightarrow \mathbb{R}^{k\times d}$
are $\mathcal{F}_{s}$-measurable given maps,
and $\psi(\omega,t):\Omega\times [0,T]\rightarrow \mathbb{R}^{k}$ is $\mathcal{F}_{T}$-measurable.
The purpose is to prove BDSVIE (\ref{55}) admits a unique solution $(Y(\cdot),Z(\cdot,\cdot))$ in $\mathcal{H}^{2}[0,T]$.
As showed in Yong \cite{Yong4}, for the sake of the uniqueness of  solutions, some additional constraints should be imposed on $Z(t,s)$ for $(t,s)\in \Delta^{c}$.
In order to do this, we introduce the symmetrical martingale solutions for BDSVIE \rf{55}.

For any fixed $t\in[0,T]$, we define
$$
 M_{1}(r)\triangleq E\big[\big(Y(t)-EY(t)\big)\mid\mathcal{F}^{W}_{r}\big], \qq
 M_{2}(\overline{r})\triangleq E\big[\big(Y(t)-EY(t)\big)\mid\mathcal{F}^{B}_{\overline{r},T}\big],
$$
where $0\leq r\leq t \leq \overline{r}\leq T$. It is easy to see that
$M_{1}(r)$ is a martingale with respect to $\mathcal{F}^{W}_{r}$ and
$M_{2}(\overline{r})$ is a backward martingale with respect to $\mathcal{F}^{B}_{\overline{r},T}$.
Then from the forward and backward martingale representation theorems (see Theorem \ref{0.5}),
there exists a unique pair
$(X_{1}(t,\cdot),X_{2}(t,\cdot))$ (parameterized by $t\in[0,T]$) belongs to
$ L^{2}_{\mathcal{F}^{W}}(0,t;\mathbb{R}^{k\times d})\times L^{2}_{\mathcal{F}^{B}}(t,T;\mathbb{R}^{k\times d})$ such that
\bel{51.1}\left\{\ba{ll}
\ds M_{1}(r)=EM_1(0)+\int\limits_0^r X_{1}(t,s) d W_{s}, \qq r\in[0,t]; \\
\ns\ds M_{2}(\overline{r})=EM_{2}(T)+\int\limits_{\overline{r}}^T X_{2}(t,s) d \overleftarrow{B}_{s}, \qq \overline{r}\in[t,T].\\
\ea\right.\ee
In particular, when $r=t$ and $\overline{r}=t$, and note that $EM_1(0)=0$ and $EM_{2}(T)=0$, we obtain
\bel{51}\left\{\ba{ll}
\ds  E[Y(t)\big| \mathcal{F}^{W}_{t}]=EY(t)+\int\limits_0^t X_{1}(t,s) d W_{s}, \qq (t,s)\in \Delta^{c}; \\
\ns\ds E[Y(t)\big| \mathcal{F}^{B}_{t,T}]=EY(t)+\int\limits_t^T X_{2}(t,s) d \overleftarrow{B}_{s}, \qq (t,s)\in \Delta.\\
\ea\right.\ee
To close the gap of $X_{1}(\cdot,\cdot)$ in $\Delta$ and the values of $X_{2}(\cdot,\cdot)$ in $\Delta^{c}$ by symmetry:
\begin{equation}\label{53}
 \begin{cases}
  X_{1}(t,s)= X_{1}(s,t), \qq (t,s)\in \Delta; \\
  X_{2}(t,s)= X_{2}(s,t), \qq (t,s)\in \Delta^{c}.\\
  \end{cases}
\end{equation}
Then
$$(X_{1}(\cdot,\cdot),X_{2}(\cdot,\cdot))\in  L^{2}_{\mathcal{F}^{W}}([0,T]^{2};\mathbb{R}^{k\times d})\times L^{2}_{\mathcal{F}^{B}}([0,T]^{2};\mathbb{R}^{k\times d}).$$
Now define the values of
$Z(\cdot,\cdot)$ on $(t,s)\in \Delta^{c}$ by:
\begin{equation}\label{54}
    Z(t,s)=  X_{1}(t,s) + X_{2}(t,s), \qq (t,s)\in \Delta^{c}.
\end{equation}
It is easy to check that when $(t,s)\in \Delta^{c}$, $X_{1}(t,s)$ is $\mathcal{F}^{W}_{s}$-adapted
and $X_{2}(t,s)$ is $\mathcal{F}^{B}_{s,T}$-adapted,
hence $Z(t,s)$ is $\mathcal{F}_{s}$-measurable for $(t,s)\in \Delta^{c}$.
\begin{definition}
A pair of $(Y(\cdot),Z(\cdot,\cdot))\in \mathcal{H}^{2}[0,T]$ is called a symmetrical martingale solution (or $\mathbf{SM}$-solution) of BDSVIE (\ref{55}),
if it satisfies (\ref{55}) in the usual It\^{o}'s sense and, in addition, (\ref{54}) holds.
\end{definition}
In the above, ``SM'' in ``SM-solution'' stands for a symmetrical martingale representation
(for $Y(t)$ to determine $Z(\cdot,\cdot)$ on $\Delta^{c}$).
Next, let $\mathcal{M}^{2}[0,T]$ be the set of all $(Y(\cdot),Z(\cdot,\cdot))\in \mathcal{H}^{2}[0,T]$ such that (\ref{54}) holds.
From (\ref{51})-(\ref{54}), one has
$$\ba{ll}
\ds 2E\int_0^T |Y(t)|^{2} dt\\
\ns\ds \geq E\int_0^T \int_0^t |X_{1}(t,s)|^{2} ds dt + E\int_0^T  \int_t^T  |X_{2}(t,s)|^{2} dsdt\\
\ns\ds = E\int_0^T \int_0^t  |X_{1}(t,s)|^{2} ds dt + E\int_0^T \int_0^t  |X_{2}(s,t)|^{2} ds dt\\
\ns\ds \geq \frac{1}{2}E\int_0^T  \int_0^t |Z(t,s)|^{2} ds dt,
\ea$$
then,
\begin{equation}\label{99}
  4E\int_0^T e^{\beta t}|Y(t)|^{2} dt \geq E\int_0^T e^{\beta t} \bigg(\int_0^t |Z(t,s)|^{2} ds\bigg) dt\geq E\int_0^T \int_0^t e^{\beta s} |Z(t,s)|^{2} ds dt.
\end{equation}
Hence we deduce the following inequality,
\begin{equation*}
       E\int_0^{T} \bigg(e^{\beta t}|Y(t)|^{2} + \int_0^{T} e^{\beta s}|Z(t,s)|^{2} ds\bigg) dt
 \leq 5E\int_0^{T} \bigg(e^{\beta t}|Y(t)|^{2} + \int_t^{T} e^{\beta s}|Z(t,s)|^{2} ds\bigg) dt.
\end{equation*}
This implies that we can use the following as an equivalent norm in $\mathcal{M}^{2}[0,T]$:
 \begin{equation*}
 \| (Y(\cdot),Z(\cdot,\cdot)) \|_{\mathcal{M}^{2}[0,T]}
 \equiv \bigg[ E\int_0^{T} e^{\beta t}|Y(t)|^{2} dt + E\int_0^{T}\int_t^{T} e^{\beta s}|Z(t,s)|^{2} ds dt \bigg]^{\frac{1}{2}}.
 \end{equation*}

\begin{remark}
In (\ref{54}), if we let
\begin{equation*}
Z(t,s)=X_{1}(t,s), \qq \forall (t,s)\in \Delta^{c},
\end{equation*}
then we can define the M-solution in the sense \cite{Yong4}.
However, here $Z(t,s)$ is $\mathcal{F}^{W}_{s}$-adapted.
Similarly, if
\begin{equation*}
Z(t,s)=Z(s,t), \qq \forall (t,s)\in \Delta^{c},
\end{equation*}
then we can also define the S-solution in the sense \cite{Wang2}.
\end{remark}

\subsection{Existence and uniqueness theorem}

Firstly, we consider the existence and uniqueness result of the following BDSVIE,
\begin{equation}\label{2}
\begin{split}
  Y(t)=&\ \psi(t) + \int_t^T f(t,s,Y(s),Z(t,s)) ds\\
       &+ \int_t^T g(t,s,Y(s),Z(t,s)) d\overleftarrow{B}_{s}- \int_t^T Z(t,s) dW_{s}, \qq t\in [0,T].
\end{split}
\end{equation}
For this type of BDSVIEs, since $f$ and $g$ are independent of $Z(s,t)$, we need not the notion of SM-solution, and just need to consider the measurable solution.

\begin{itemize}
  \item [$\mathbf{(H2)}$]
  Assume
\begin{equation*}
f:\Omega\times \Delta\times \mathbb{R}^{k}\times \mathbb{R}^{k\times d}\rightarrow \mathbb{R}^{k}; \ \
g:\Omega\times \Delta\times \mathbb{R}^{k}\times \mathbb{R}^{k\times d}\rightarrow \mathbb{R}^{k\times d},
\end{equation*}
be jointly measurable such that for all $(y,z)\in \mathbb{R}^{k}\times \mathbb{R}^{k\times d}$,
\begin{equation*}
f(\cdot,\cdot,y,z)\in L_{\mathcal{F}}^{2}(\Delta; \mathbb{R}^{k}), \ \
g(\cdot,\cdot,y,z)\in L_{\mathcal{F}}^{2}(\Delta; \mathbb{R}^{k\times d}).
\end{equation*}
Furthermore, there exist constants $c>0$ and $0<\alpha<\frac{1}{2}$ such that for any
$y,y'\in\mathbb{R}^{k},z,z'\in {R}^{k\times d}$ and
$(t,s) \in \Delta$,
\begin{equation*}
 \begin{cases}
   |f(t,s,y,z)-f(t,s,y',z')|^{2}\leq c(|y-y'|^{2} + \|z-z'\|^{2});\\
   \|g(t,s,y,z)-g(t,s,y',z')\|^{2}\leq \alpha(|y-y'|^{2} + \|z-z'\|^{2}).\\
 \end{cases}
\end{equation*}
\end{itemize}

\begin{theorem}\label{19}
Let $\psi\in L_{\mathcal{F}_{T}}^2(0,T;\mathbb{R}^{k})$, and suppose $f$ and $g$ satisfy the assumption (H2).
Then  BDSVIE (\ref{2}) has unique solution $(Y(\cdot),Z(\cdot,\cdot))$ $\in \mathcal{H}_{\Delta}^{2}[0,T]$,
 and the following estimate holds,
\bel{21}\ba{ll}
\ds E\int_0^T e^{\beta t} |Y(t)|^{2} dt  +  E\int_0^T\int_t^T  e^{\beta s} \|Z(t,s)\|^{2} dsdt \\
\ns\ds\leq  K\bigg[E \int_0^T e^{\beta t}|\psi(t)|^{2} dt + E\int_0^T \int_t^T e^{\beta s} |f(t,s,0,0)|^{2} dsdt
+E\int_0^T \int_t^T e^{\beta s} \|g(t,s,0,0)\|^{2} dsdt\bigg],
\ea\ee
where $K$ is a positive constant which may be different from line to line.
\end{theorem}

\begin{proof}
The method used in the proof, similar in \cite{Wang},
 is inspired by the method of estimating the adapted solutions of BSDEs in \cite{El}.
We introduce the following family of BDSDEs (parameterized by $t\in [0,T]$), for $r\in [t,T],$
\begin{equation}\label{5}
   \lambda(r)=\psi(t) + \int_r^T f(t,s,\lambda(s),\mu(t,s)) ds
    + \int_r^T g(t,s,\lambda(s),\mu(t,s)) d\overleftarrow{B}_{s} - \int_r^T \mu(t,s) dW_{s}.
\end{equation}
By Proposition \ref{3}, Eq. (\ref{5}) admits a unique solution
$(\lambda(\cdot),\mu(t,\cdot))$ on $[t,T]$,  $\forall t\in [0,T]$.
Let
\begin{equation*}
Y(s)=\lambda(s), \qq Z(t,s)=\mu(t,s), \qq 0\leq t\leq s\leq T.
\end{equation*}
Then from (\ref{5}), for $r\in [t,T]$,
\bel{42}Y(r)+\int_r^T Z(t,s) dW_{s} =\psi(t) + \int_r^T f(t,s,Y(s),Z(t,s)) ds
    + \int_r^T g(t,s,Y(s),Z(t,s)) d\overleftarrow{B}_{s}. \ee
Especially when $r=t$, we obtain that
\begin{equation}\label{57}
Y(t)+\int_t^T Z(t,s) dW_{s} =\psi(t) + \int_t^T f(t,s,Y(s),Z(t,s)) ds
    + \int_t^T g(t,s,Y(s),Z(t,s)) d\overleftarrow{B}_{s}.
\end{equation}
Hence $(Y(\cdot),Z(\cdot,\cdot))$ is a solution of (\ref{2}).
Now we estimate
\begin{equation*}
   E\int_0^T e^{\beta t} |Y(t)|^{2} dt + E\int_0^T \int_t^T e^{\beta s} \|Z(t,s)\|^{2} dsdt.
\end{equation*}
In the following, for notational simplicity, we denote
$$\widetilde{f}(t,s)=f(t,s,Y(s),Z(t,s)), \qq \widetilde{g}(t,s)=g(t,s,Y(s),Z(t,s)).$$
By Cauchy-Schwarz inequality, it follows that
\bel{6}\ba{ll}
\ds \bigg|\int_s^T \widetilde{f}(t,u) du\bigg|^{2}\\
\ns\ds= \bigg|\int_s^T e^{\frac{-\gamma u}{2}} e^{\frac{\gamma u}{2}} \widetilde{f}(t,u) du\bigg|^{2}  \nonumber\\
\ns\ds\leq \int_s^T e^{-\gamma u}  du \cdot \int_s^T e^{\gamma u} |\widetilde{f}(t,u)|^{2} du  \nonumber\\
\ns\ds\leq \frac{1}{\gamma} e^{-\gamma s} \int_s^T e^{\gamma u} |\widetilde{f}(t,u)|^{2} du, \qq 0\leq s\leq t\leq T,  \ea\ee
where $\gamma=\frac{\beta}{2}$ or $\beta$. By taking $\gamma=\frac{\beta}{2}$ in (\ref{6}), we see that
$$\ba{ll}
\ds  \int_t^T \beta e^{\beta s}  \bigg|\int_s^T \widetilde{f}(t,u) du\bigg|^{2} ds\\
\ns\ds\leq \frac{4}{\beta} \int_t^T \frac{\beta}{2} e^{\frac{\beta}{2} s}
        \cdot\bigg(\int_s^T e^{\frac{\beta}{2} u} |\widetilde{f}(t,u)|^{2} du\bigg) ds\\
\ns\ds\leq \bigg(\frac{4}{\beta}e^{\frac{\beta}{2} s}
      \int_s^T e^{\frac{\beta}{2} u} |\widetilde{f}(t,u)|^{2} du\bigg) \bigg|_{t}^{T}
        +\frac{4}{\beta} \int_t^T e^{\beta s} |\widetilde{f}(t,s)|^{2} ds\\
\ns\ds\leq \frac{4}{\beta} \int_t^T e^{\beta s} |\widetilde{f}(t,s)|^{2} ds.
\ea$$
Therefore
\begin{equation}\label{7}
  E\int_0^T \int_t^T \beta e^{\beta s} \bigg|\int_s^T \widetilde{f}(t,u) du\bigg|^{2} dsdt
  \leq \frac{4}{\beta}  E\int_0^T \int_t^T e^{\beta s} |\widetilde{f}(t,s)|^{2} dsdt.
\end{equation}
We also obtain the following result by taking $s=t$ and $\gamma=\beta$ in (\ref{6}),
\begin{equation}\label{8}
  E\int_0^T e^{\beta t} \bigg|\int_t^T \widetilde{f}(t,s) ds\bigg|^{2} dt
  \leq \frac{1}{\beta}  E\int_0^T \int_t^T e^{\beta s} |\widetilde{f}(t,s)|^{2} dsdt.
\end{equation}
On the other hand, since
$$\ba{ll}
\ds   E\int_t^T \beta e^{\beta s}\bigg(\int_s^T \|\widetilde{g}(t,u)\|^{2} du\bigg) ds\\
\ns\ds = E\bigg( e^{\beta s} \int_s^T \|\widetilde{g}(t,u)\|^{2}du \bigg) \bigg|_{t}^{T} + E \int_t^T e^{\beta s} \|\widetilde{g}(t,s)\|^{2} ds\\
\ns\ds \leq E \int_t^T e^{\beta s} \|\widetilde{g}(t,s)\|^{2} ds,
\ea$$
we obtain
\begin{equation}\label{10}
  E\int_0^T \int_t^T \beta e^{\beta s}\bigg(\int_s^T \|\widetilde{g}(t,u)\|^{2} du\bigg) dsdt
  \leq  E\int_0^T \int_t^T e^{\beta s} \|\widetilde{g}(t,s)\|^{2} dsdt.
\end{equation}
Similarly, it's easy to see that
\begin{equation}\label{11}
    \int_r^T \beta e^{\beta s}\bigg(\int_s^T \|Z(t,u)\|^{2} du\bigg) ds
   =\bigg(e^{\beta s}\int_s^T \|Z(t,u)\|^{2} du\bigg)\bigg|_{r}^{T} + \int_r^T  e^{\beta s} \|Z(t,s)\|^{2} ds, \qq r\in [t,T].
\end{equation}
For every $ t\in [0,T]$, we can rewrite (\ref{11}) after taking $r=t$,
\bel{12}\ba{ll}
\ds E\int_0^T \int_t^T  e^{\beta s}\|Z(t,s)\|^{2} dsdt\\
\ns\ds = E\int_0^T\int_t^T \beta e^{\beta s}\bigg(\int_s^T \|Z(t,u)\|^{2} du\bigg) dsdt + E\int_0^T e^{\beta t}\int_t^T \|Z(t,u)\|^{2} du dt.
\ea\ee
Notice $\psi(t)$ is $\mathcal{F}_{T}$-measurable, then by using the property of conditional expectation,
it follows from (\ref{57}) that
$$\ba{ll}
\ds  E|Y(t)|^{2}+ E\int_t^T \|Z(t,s)\|^{2} ds\\
\ns\ds = E\bigg(\psi(t) + \int_t^T \widetilde{f}(t,s) ds + \int_t^T \widetilde{g}(t,s) d\overleftarrow{B}_{s}\bigg)^{2}\\
\ns\ds \leq E\bigg(2|\psi(t)|^{2} + 2\bigg|\int_t^T \widetilde{f}(t,s) ds\bigg|^{2} + \int_t^T \|\widetilde{g}(t,s)\|^{2} ds\bigg).
\ea$$
Therefore by (\ref{8}),
\bel{15}\ba{ll}
\ds  E\int_0^T e^{\beta t}|Y(t)|^{2}dt + E\int_0^T e^{\beta t}\int_t^T \|Z(t,s)\|^{2} dsdtr\\
\ns\ds\leq E \int_0^T\bigg(2 e^{\beta t}|\psi(t)|^{2} + \frac{2}{\beta}\int_t^T e^{\beta s} |\widetilde{f}(t,s)|^{2} ds+ e^{\beta t} \int_t^T \|\widetilde{g}(t,s)\|^{2} ds\bigg)dt.
\ea\ee
Similarly, from (\ref{42}), we obtain
\begin{equation*}
     E|Y(s)|^{2}+ E\int_s^T \|Z(t,u)\|^{2} du
\leq E\bigg(2|\psi(t)|^{2} + 2\bigg|\int_s^T \widetilde{f}(t,u) du\bigg|^{2} + \int_s^T \|\widetilde{g}(t,u)\|^{2} du\bigg).
\end{equation*}
Then, from (\ref{7}) and (\ref{10}),
\bel{14}\ba{ll}
\ds  E\int_0^T\int_t^T \beta e^{\beta s}\bigg(\int_s^T \|Z(t,u)\|^{2} du\bigg) dsdt\\
\ns\ds \leq E\int_0^T\bigg( 2e^{\beta t}|\psi(t)|^{2} + \frac{8}{\beta}\int_t^T e^{\beta s}|\widetilde{f}(t,s)|^{2} ds + \int_t^T e^{\beta s}\|\widetilde{g}(t,s)\|^{2} ds \bigg)dt.
\ea\ee
Hence by combining (\ref{12}), (\ref{15}) and (\ref{14}), it follows that
\bel{0}\ba{ll}
\ds  E\int_0^T e^{\beta t} |Y(t)|^{2} dt  +  E\int_0^T\int_t^T e^{\beta s} \|Z(t,s)\|^{2} dsdt\\
\ns\ds \leq  4E \int_0^Te^{\beta t}|\psi(t)|^{2} dt
     + \frac{10}{\beta}  E\int_0^T \int_t^T e^{\beta s} |\widetilde{f}(t,s)|^{2} dsdt \\
\ns\ds+ E\int_0^T e^{\beta t} \int_t^T \|\widetilde{g}(t,s)\|^{2} ds dt
     + E\int_0^T \int_t^T e^{\beta s} \|\widetilde{g}(t,s)\|^{2} dsdt.
\ea\ee
Or
$$\ba{ll}
\ds  E\int_0^T e^{\beta t} |Y(t)|^{2} dt  +  E\int_0^T\int_t^T e^{\beta s} \|Z(t,s)\|^{2} dsdt \\
\ns\ds \leq  4E \int_0^Te^{\beta t}|\psi(t)|^{2} dt
     + \frac{10}{\beta}  E\int_0^T \int_t^T e^{\beta s} |f(t,s,Y(s),Z(t,s))|^{2} dsdt \\
\ns\qq \ \ \ + 2E\int_0^T \int_t^T e^{\beta s} \|g(t,s,Y(s),Z(t,s))\|^{2} dsdt\\
\ns\ds\leq  4E \int_0^Te^{\beta t}|\psi(t)|^{2} dt
     + \frac{10}{\beta}  E\int_0^T \int_t^T e^{\beta s} |f(t,s,0,0)|^{2} dsdt
     + 2E\int_0^T \int_t^T e^{\beta s} \|g(t,s,0,0)\|^{2} dsdt\\
\ns\qq \ \ \ + (\frac{10c}{\beta}+2\alpha)
    E\bigg[\int_0^T e^{\beta t} |Y(t)|^{2} dt  +  E\int_0^T\int_t^T e^{\beta s} \|Z(t,s)\|^{2} dsdt \bigg].
\ea$$
Hence
$$\ba{ll}
\ds  \big(1-\frac{10c}{\beta}-2\alpha\big)E\bigg[\int_0^T e^{\beta t} |Y(t)|^{2} dt
+  E\int_0^T\int_t^T e^{\beta s} \|Z(t,s)\|^{2} dsdt \bigg] \\
\ns\ds \leq 4E \int_0^Te^{\beta t}|\psi(t)|^{2} dt
+ \frac{10}{\beta}  E\int_0^T \int_t^T e^{\beta s} |f(t,s,0,0)|^{2} dsdt
+ 2E\int_0^T \int_t^T e^{\beta s} \|g(t,s,0,0)\|^{2} dsdt.
\ea$$
Now by letting $\beta=\frac{10c}{1-2\alpha}+1$, then the estimate (\ref{21}) holds, which implies the uniqueness of BDSVIE (\ref{2}).
\end{proof}

From the above theorem, a corollary follows directly.

\begin{corollary}\label{9}
Let $\psi\in L_{\mathcal{F}_{T}}^2(0,T;\mathbb{R}^{k})$,
 $f\in L_{\mathcal{F}}^2(\Delta;\mathbb{R}^{k})$ and $g\in L_{\mathcal{F}}^2(\Delta;\mathbb{R}^{k\times d})$.
Then  BDSVIE:
\begin{equation*}
   Y(t)=\psi(t) + \int_t^T f(t,s) ds + \int_t^T g(t,s) d\overleftarrow{B}_{s} - \int_t^T Z(t,s) dW_{s}, \qq t\in [0,T],
\end{equation*}
has unique solution $(Y(\cdot),Z(\cdot,\cdot))$ $\in \mathcal{H}_{\Delta}^{2}[0,T]$, and the following estimate holds,
\bel{20}\ba{ll}
\ds E\int_0^T e^{\beta t} |Y(t)|^{2} dt  +  E\int_0^T\int_t^T  e^{\beta s} \|Z(t,s)\|^{2} dsdt \\
\ns\ds\leq 4E \int_0^T e^{\beta t}|\psi(t)|^{2} dt + \frac{10}{\beta}  E\int_0^T \int_t^T e^{\beta s} |f(t,s)|^{2} dsdt \\
\ns\ds + E\int_0^T e^{\beta t} \int_t^T \|g(t,s)\|^{2} ds dt +E\int_0^T \int_t^T e^{\beta s} \|g(t,s)\|^{2} dsdt. \ea\ee
\end{corollary}

\begin{remark}
The estimate (\ref{20}) follows from (\ref{0}).
A detailed proof of the above corollary is presented in Shi and Wen \cite{Wen}.
\end{remark}

\begin{remark}\label{98}
Under the assumptions of Theorem \ref{19}, if $Z(\cdot,\cdot)$ on $\Delta^{c}$  is defined as follows,
\begin{equation*}
  Z(t,s)= X_{1}(t,s) + X_{2}(t,s), \qq \forall (t,s)\in \Delta^{c},
\end{equation*}
where $X_{1} \in L^{2}_{\mathcal{F^{W}}}([0,T]^{2};\mathbb{R}^{k\times d})$, $X_{2} \in L^{2}_{\mathcal{F^{B}}}([0,T]^{2};\mathbb{R}^{k\times d})$
are determined by:
$$\left\{\ba{ll}
\ds E[Y(t)\big| \mathcal{F}^{W}_{t}]=EY(t)+\int\limits_0^t X_{1}(t,s) d W_{s}, \qq (t,s)\in \Delta^{c}; \\
\ns\ds E[Y(t)\big| \mathcal{F}^{B}_{t,T}]=EY(t)+\int\limits_t^T X_{2}(t,s) d \overleftarrow{B}_{s}, \qq (t,s)\in \Delta,
\ea\right.$$
and
\begin{equation*}
 \begin{cases}
  X_{1}(t,s)= X_{1}(s,t), \qq (t,s)\in \Delta; \\
  X_{2}(t,s)= X_{2}(s,t), \qq (t,s)\in \Delta^{c}.\\
  \end{cases}
\end{equation*}
Then BDSVIE  (\ref{2}) admits a unique SM-solution in $\mathcal{H}^{2}[0,T]$.
\end{remark}

Now we prove the existence and uniqueness of SM-solutin of BDSVIE (\ref{55}).

\begin{itemize}
  \item [$\mathbf{(H3)}$]
  Assume
\begin{equation*}
f:\Omega\times \Delta\times \mathbb{R}^{k}\times \mathbb{R}^{k\times d}\times \mathbb{R}^{k\times d}\rightarrow \mathbb{R}^{k}; \ \
g:\Omega\times \Delta\times \mathbb{R}^{k}\times \mathbb{R}^{k\times d}\times \mathbb{R}^{k\times d}\rightarrow \mathbb{R}^{k\times d},
\end{equation*}
be jointly measurable such that for all $(y,z,\zeta)\in \mathbb{R}^{k}\times \mathbb{R}^{k\times d}\times \mathbb{R}^{k\times d}$,
\begin{equation*}
f(\cdot,\cdot,y,z,\zeta)\in L_{\mathcal{F}}^{2}(\Delta; \mathbb{R}^{k}), \ \
g(\cdot,\cdot,y,z,\zeta)\in L_{\mathcal{F}}^{2}(\Delta; \mathbb{R}^{k\times d}).
\end{equation*}
Furthermore, there exist constants $c>0$ and $0<\alpha<\frac{1}{T+8}$ such that for any  $y,y'\in\mathbb{R}^{k},z,z',\zeta,\zeta' \in {R}^{k\times d}$ and
$(t,s) \in \Delta$,
\begin{equation*}
 \begin{cases}
   |f(t,s,y,z,\zeta)-f(t,s,y',z',\zeta')|^{2}\leq c(|y-y'|^{2} + \|z-z'\|^{2} + \|\zeta-\zeta'\|^{2});\\
   \|g(t,s,y,z,\zeta)-g(t,s,y',z',\zeta')\|^{2}\leq \alpha(|y-y'|^{2} + \|z-z'\|^{2} + \|\zeta-\zeta'\|^{2}).\\
 \end{cases}
\end{equation*}
\end{itemize}

For notational simplicity, in the following we denote $f_{0}(t,s)=f(t,s,0,0,0)$ and $g_{0}(t,s)=g(t,s,0,0,0)$.
Now we establish the main result.

\begin{theorem}\label{22}
  Under the assumption (H3), for any $\psi(\cdot)\in L_{\mathcal{F}_{T}}^2(0,T;\mathbb{R}^{k})$,
  BDSVIE (\ref{55}) admits unique SM-solution $(Y(\cdot),Z(\cdot,\cdot))\in \mathcal{H}^{2}[0,T]$.
  Moreover, the following estimate holds,
\bel{18}\ba{ll}
\ds E\bigg[\int_0^T e^{\beta t} |Y(t)|^{2} dt + \int_0^T \int_t^T e^{\beta s}  \|Z(t,s)\|^{2} dsdt\bigg]\\
\ns\ds \leq KE\bigg[\int_0^T |\psi(t)|^{2} dt +  \int_0^T \int_t^T e^{\beta s}|f_{0}(t,s)|^{2} dsdt
+\int_0^T \int_t^T e^{\beta s}\|g_{0}(t,s)\|^{2} dsdt\bigg].
\ea\ee
\end{theorem}

\begin{proof}
For any $(y(\cdot),z(\cdot,\cdot))\in \mathcal{M}^{2}[0,T]$, consider the following BDSVIE,
\begin{equation}\label{23}
\begin{split}
   Y(t)=&\psi(t) + \int_t^T f(t,s,y(s),z(t,s),z(s,t)) ds \\
        & + \int_t^T g(t,s,y(s),z(t,s),z(s,t)) d\overleftarrow{B}_{s} - \int_t^T Z(t,s) dW_{s}, \qq t\in[0,T].
\end{split}
\end{equation}
From Corollary \ref{9}, Eq. (\ref{23}) exists unique solution $(Y(\cdot),Z(\cdot,\cdot))\in \mathcal{H}_{\Delta}^{2}[0,T]$.
Now we define $Z(\cdot,\cdot)$ on $\Delta^{c}$ as follows:
\begin{equation*}
  Z(t,s)= X_{1}(t,s) + X_{2}(t,s), \qq \forall (t,s)\in \Delta^{c},
\end{equation*}
where $X_{1} \in L^{2}_{\mathcal{F^{W}}}([0,T]^{2};\mathbb{R}^{k\times d})$ and $X_{2} \in L^{2}_{\mathcal{F^{B}}}([0,T]^{2};\mathbb{R}^{k\times d})$
are determined by:
$$\left\{\ba{ll}
\ds E[Y(t)\big| \mathcal{F}^{W}_{t}]=EY(t)+\int\limits_0^t X_{1}(t,s) d W_{s}, \qq (t,s)\in \Delta^{c}; \\
\ns\ds E[Y(t)\big| \mathcal{F}^{B}_{t,T}]=EY(t)+\int\limits_t^T X_{2}(t,s) d \overleftarrow{B}_{s}, \qq (t,s)\in \Delta,
\ea\right.$$
and
\begin{equation*}
 \begin{cases}
  X_{1}(t,s)= X_{1}(s,t), \qq (t,s)\in \Delta; \\
  X_{2}(t,s)= X_{2}(s,t), \qq (t,s)\in \Delta^{c}.\\
  \end{cases}
\end{equation*}
 Then $(Y(\cdot),Z(\cdot,\cdot))\in \mathcal{M}^{2}[0,T]$ is the unique SM-solution to BDSVIE (\ref{23}),
and we can define a mapping $\Theta:\mathcal{M}^{2}[0,T]\rightarrow \mathcal{M}^{2}[0,T]$ by
\begin{equation*}
\Theta(y(\cdot),z(\cdot,\cdot))=(Y(\cdot),Z(\cdot,\cdot)), \qq \forall (y(\cdot),z(\cdot,\cdot))\in \mathcal{M}^{2}[0,T].
\end{equation*}
 From the estimate (\ref{20}),
\bel{24}\ba{ll}
\ds E\int_0^T e^{\beta t} |Y(t)|^{2} dt + E\int_0^T \int_t^T e^{\beta s} |Z(t,s)|^{2} dsdt\\
\ns\ds \leq 4e^{\beta T}E\int_0^T |\psi(t)|^{2} dt
+ \frac{10}{\beta} E\int_0^T \int_t^T e^{\beta s}|f(t,s,y(s),z(t,s),z(s,t))|^{2} dsdt\\
\ns\ds + E\int_0^T e^{\beta t}\int_t^T|g(t,s,y(s),z(t,s),z(s,t))|^{2} dsdt \\
\ns\ds+ E\int_0^T \int_t^T e^{\beta s}|g(t,s,y(s),z(t,s),z(s,t))|^{2} dsdt.
\ea\ee
For the second term in the right part of (\ref{24}), from (H2) and notice (\ref{99}), one has
$$\ba{ll}
\ds E \int_0^T \int_t^T e^{\beta s} |f(t,s,y(s),z(t,s),z(s,t))|^{2} dsdt\\
\ns\ds \leq E \int_0^T \int_t^T e^{\beta s}\left(|f_{0}(t,s)|^{2} + c|y(s)|^{2} + c|z(t,s)|^{2} + c|z(s,t)|^{2}\right) dsdt\\
\ns\ds \leq c(T+4)E \int_0^T e^{\beta t}|y(t)|^{2} dt+ cE\int_0^T \int_t^T e^{\beta s}|z(t,s)|^{2} ds dt
 +E \int_0^T \int_t^T e^{\beta s}|f_{0}(t,s)|^{2} dsdt.
\ea$$
Similarly, for the third term in the right part of (\ref{24}),
$$\ba{ll}
\ds  E\int_0^T e^{\beta t}\int_t^T|g(t,s,y(s),z(t,s),z(s,t))|^{2} dsdt\\
\ns\ds \leq E \int_0^T e^{\beta t} \int_t^T\left(|g_{0}(t,s)|^{2} + \alpha|y(s)|^{2} + \alpha|z(t,s)|^{2} + \alpha|z(s,t)|^{2}\right) dsdt\\
\ns\ds \leq (\frac{\alpha}{\beta}+4\alpha) E \int_0^T e^{\beta t}|y(t)|^{2} dt
        + \alpha E\int_0^T \int_t^T e^{\beta s}|z(t,s)|^{2} ds dt
        + E \int_0^T \int_t^T e^{\beta s}|g_{0}(t,s)|^{2} dsdt.
\ea$$
Also, for the fourth term in the right part of (\ref{24}),
$$\ba{ll}
\ds  E \int_0^T \int_t^T e^{\beta s} |g(t,s,y(s),z(t,s),z(s,t))|^{2} dsdt\\
\ns\ds \leq E \int_0^T \int_t^T e^{\beta s}\left(|g_{0}(t,s)|^{2} + \alpha|y(s)|^{2} + \alpha|z(t,s)|^{2} + \alpha|z(s,t)|^{2}\right) dsdt\\
\ns\ds\leq \alpha(T+4)E \int_0^T e^{\beta t}|y(t)|^{2} dt
        + \alpha E\int_0^T \int_t^T e^{\beta s}|z(t,s)|^{2} ds dt
        +E \int_0^T \int_t^T e^{\beta s}|g_{0}(t,s)|^{2} dsdt.
\ea$$
Hence we deduce
\bel{25}\ba{ll}
\ds  E\int_0^T e^{\beta t} |Y(t)|^{2} dt + E\int_0^T \int_t^T e^{\beta s} |Z(t,s)|^{2} dsdt \\
\ns\ds \leq 4e^{\beta T}E\int_0^T |\psi(t)|^{2} dt
+ \frac{10}{\beta}E \int_0^T \int_t^T e^{\beta s}|f_{0}(t,s)|^{2} dsdt
 +2 E \int_0^T \int_t^T e^{\beta s}|g_{0}(t,s)|^{2} dsdt  \\
\ns\ds +  [\frac{K}{\beta}+\alpha(T+8) ]E \int_0^T e^{\beta t}|y(t)|^{2} dt
         + (\frac{K}{\beta}+2\alpha)E\int_0^T \int_t^T e^{\beta s}|z(t,s)|^{2} ds dt.
\ea\ee
Now, if $(Y_{i}(\cdot),Z_{i}(\cdot,\cdot))$ is the corresponding SM-solution of $(y_{i}(\cdot),z_{i}(\cdot,\cdot))$ to BDSVIE (\ref{23}), $i=1,2$, then
$$\ba{ll}
\ds E \bigg(\int_0^{T}e^{\beta t}|Y_{1}(t)-Y_{2}(t)|^{2} dt
   + \int_0^{T}\int_t^{T} e^{\beta s}|Z_{1}(t,s)-Z_{2}(t,s)|^{2} ds dt\bigg)\\
\ns\ds \leq\big[ \frac{K}{\beta}+\alpha(T+8) \big]
       E \bigg(\int_0^{T}e^{\beta t}|y_{1}(t)-y_{2}(t)|^{2} dt + \int_0^{T}\int_t^{T} e^{\beta s}|z_{1}(t,s)-z_{2}(t,s)|^{2} ds dt\bigg).
\ea$$
Let $\beta=\frac{2K}{1-\alpha(T+8)}$,
then the mapping $\Theta$ is contractive on $\mathcal{H}^{2}[0,T]$,
which implies BDSVIE (\ref{55}) admits a unique SM-solution.
And the estimate (\ref{18}) directly follows from  (\ref{25}).
This completes the proof.
\end{proof}

\section{Appendix}

To our best knowledge, since we haven't find a detail proof for the backward martingale representation theorem, in the following we present its proof in details. In fact, similar to the proof of classical martingale representation theorem, it is not difficult to prove the backward martingale representation theorem.

First, we present two lemmas which will be used in the following. The first lemma is the bakcward It\^{o} formula and the second lemma is a basic property of the space $L^2(\cF_{0,T}^B,\dbP)$.

\begin{lemma}[Pardoux-Peng \cite{Peng}, Lemma 1.3]\label{0.1}
Let $\a\in S_{\mathcal{F}^B}^2(0,T;\dbR)$, $\b,\g\in L_{\mathcal{F}^B}^2(0,T;\dbR)$ be such that,
$$\a_t=\a_0+\int_0^t\b_sds+\int_0^t\g_sd\overleftarrow{B}_s,\qq~0\leq~t\leq~T.$$
Then, if $\p\in C^2(\dbR)$, we have
$$\ds\p(\a_t)=\p(\a_0)+\int_0^t\p'(\a_s)\b_sds+\int_0^t\p'(\a_s)\g_sd\overleftarrow{B}_s-\int_0^t\p''(\a_s)\g^2_sds.$$
\end{lemma}

\begin{lemma}[{\O}ksendal \cite{Oksendal2003}, Lemma 4.3.2]\label{0.2}
The linear space of random variables of the type
\bel{0.2.2}\exp\bigg\{\int_0^Th_td\overleftarrow{B}_t
-\frac{1}{2}\int_0^Th^2_tdt\bigg\};\qq h\in L^2[0,T]~(deterministic)\ee
is dense in $L^2(\cF_{0,T}^B,\dbP)$.
\end{lemma}

\begin{remark} (i)  Lemma \ref{0.1} is a simplified version of Lemma 1.3 of Pardoux-Peng \cite{Peng}.

(ii) For Lemma \ref{0.2}, note that the integrand $h$ in the backward It\^{o} integral \rf{0.2.2} is a deterministic function. In fact, when the integrand $h$ is a deterministic function, the forward It\^{o} integral and backward It\^{o} integral are coincide (see e.g., Pardoux-Protter \cite{PP1987}). Moreover, the filtration $\cF_{0,T}^B$ is the same as $\cF_{T}^B$ from the definition. Therefore, Lemma \ref{0.2} is the same as Lemma 4.3.2 of {\O}ksendal \cite{Oksendal2003}, and the proof of
Lemma \ref{0.2} is certainly the same as the proof of Lemma 4.3.2 of {\O}ksendal \cite{Oksendal2003}.
\end{remark}

Now we prove the backward It\^{o} martingale representation theorem.

\begin{theorem}\label{0.3}
Let $F\in L^2(\cF_{0,T}^B,\dbP)$. Then there exists a unique stochastic process
$f\in L^2_{\cF^{B}}(\Om\ts[0,T];\dbR)$ such that
\bel{0.4}F=\dbE[F]+\int_0^Tf_td\overleftarrow{B}_t.\ee
\end{theorem}

\begin{proof}
First assume that $F$ has the form \rf{0.2.2}, i.e.,
$$F(\o)=\exp\bigg\{\int_0^Th_td\overleftarrow{B}_t(\o)
-\frac{1}{2}\int_0^Th^2_tdt\bigg\};\qq h\in L^2[0,T]$$
for some $h\in L^2[0,T]$ (deterministic). Define
$$Y_t(\o)=\exp\bigg\{\int_t^Th_td\overleftarrow{B}_t(\o)
-\frac{1}{2}\int_t^Th^2_tdt\bigg\};\qq 0\leq t\leq T.$$
Then by It\^{o}'s formula (Lemma \ref{0.1})
$$dY_t=Y_t\big(-h_td\overleftarrow{B}_t+\frac{1}{2}h^2_tdt\big)-\frac{1}{2}Y_t\big(-h_td\overleftarrow{B}_t\big)^2
=-Y_th_td\overleftarrow{B}_t.$$
So that
$$Y_T=Y_t-\int_t^TY_sh_sd\overleftarrow{B}_s,$$
or
$$\ba{ll}
\ds Y_t=Y_T+\int_t^TY_sh_sd\overleftarrow{B}_s\\
\ns\ds\ \ \ \ =1+\int_t^TY_sh_sd\overleftarrow{B}_s.
\ea$$
Therefore
$$F=Y_0=1+\int_0^TY_sh_sd\overleftarrow{B}_s,$$
and hence $\dbE[F]=1$. So \rf{0.4} holds in this case.

\ms

If $F\in L^2(\cF_{0,T}^B,\dbP)$ is arbitrary, we can by Lemma \ref{0.2} approximate $F$ in $L^2(\cF_{0,T}^B,\dbP)$ by linear combinations $F_n$ of functions of the form \rf{0.2.2}. Then for each $n$ we have
$$F_n=\dbE[F_n]+\int_0^Tf_n(t)d\overleftarrow{B}(t),\qq where \ \ f_n\in\cV(0,T).$$
By the It\^{o} isometry
$$\ba{ll}
\ds \dbE\big[(F_n-F_m)^2\big]=\dbE\bigg[ \bigg(\dbE[F_n-F_m]+\int_0^T(f_n(t)-f_m(t))d\overleftarrow{B}(t)\bigg)^2\bigg]\\
\ns\ds\qq\qq\qq \ \ =\big(\dbE[F_n-F_m]\big)^2+\int_0^T\dbE[(f_n(t)-f_m(t))^2]dt
\rightarrow 0\qq as \ \ n,m\rightarrow \i.
\ea$$
So $\{f_n\}$ is a Cauchy sequence in $L^2([0,T]\ts\Omega)$ and hence converges to some $f\in L^2([0,T]\ts\Omega)$.
Since $f_n\in\cV(0,T)$ we have  $f\in\cV(0,T)$. Again using the It\^{o} isometry we see that
$$F=\lim_{n\ra \i}F_n=\lim_{n\ra \i}\bigg(\dbE[F_n]+\int_0^Tf_nd\overleftarrow{B}\bigg)
=\dbE[F]+\int_0^Tfd\overleftarrow{B},$$
the limit being taken in $L^2(\cF_{0,T}^B,\dbP)$. Hence the representation \rf{0.4} holds for all
$F\in L^2(\cF_{0,T}^B,\dbP)$.

\ms

The uniqueness follows from the It\^{o} isometry: Suppose
$$F=\dbE[F]+\int_0^Tf_1(t)d\overleftarrow{B}(t)=\dbE[F]+\int_0^Tf_2(t)d\overleftarrow{B}(t)$$
with $f_1,f_2\in\cV(0,T)$. Then
$$0=\dbE\bigg[\bigg(\int_0^T\Big(f_1(t)-f_2(t)\Big)d\overleftarrow{B}(t)\bigg)^2\bigg]
=\int_0^T\dbE\Big[\big(f_1(t)-f_2(t)\big)^2\Big]dt$$
and therefore $f_1(t,\o)=f_2(t,\o)$ for a.a. $(t,\o)\in[0,T]\ts\Om$.
\end{proof}

\begin{theorem}\label{0.5}
Suppose $M_t$ is an $\cF^{B}_{t,T}$-martingale and that $M_t\in L^2(\dbP)$ for all $0\leq t\leq T$.
Then there exists a unique stochastic process $f\in L^2_{\cF^{B}}(\Om\ts[0,T];\dbR)$ such that
\bel{0.6}M_t=\dbE[M_T]+\int_t^Tf_sd\overleftarrow{B}_s,\qq a.s., \ for \ all \ t\in[0,T].\ee
\end{theorem}

\begin{proof}
By Theorem \ref{0.3} applied to $F=M_t$, we have that for all $t\in[0,T]$,
there exists a unique $f^t(s,\o)\in L^2(\cF_{t,T}^B,\dbP)$ such that
$$M(t)=\dbE[M(t)]+\int_t^Tf^t(s)d\overleftarrow{B}(s)
      =\dbE[M(T)]+\int_t^Tf^t(s)d\overleftarrow{B}(s).$$
Now assume $0\leq t_1\leq t_2\leq T$. Then
$$\ba{ll}
\ds M(t_2)=\dbE\Big[M_{t_1}\Big|\cF^B_{t_2,T}\Big]\\
\ns\ds\qq \ \ =\dbE[M(T)]+\dbE\Big[\int_{t_1}^Tf^{t_1}(s)d\overleftarrow{B}(s)\Big|\cF^B_{t_2,T}\Big]\\
\ns\ds\qq \ \ =\dbE[M(T)]+\int_{t_2}^Tf^{t_1}(s)d\overleftarrow{B}(s).
\ea$$
But we also have
$$M(t_2)=\dbE[M(T)]+\int_{t_2}^Tf^{t_2}(s)d\overleftarrow{B}(s).$$
Hence, by comparing the above two equations we get that
$$0=\dbE\bigg[\bigg(\int_{t_2}^T\big(f^{t_1}(s)-f^{t_2}(s)\big)d\overleftarrow{B}(s)\bigg)^2\bigg]
   =\int_{t_2}^T\dbE\big[\big(f^{t_1}(s)-f^{t_2}(s)\big)^2\big]ds$$
and therefore
$$f^{t_1}(s,\o)=f^{t_2}(s,\o)\qq for \ a.a. \ (s,\o)\in [t_2,T]\ts\Om.$$
So we can define $f(s,\o)$ for a.a. $(s,\o)\in [0,T]\ts\Om$ by setting
$$f(s,\o)=f^N(s,\o)\qq if \ s\in[N,T],$$
and then we get
$$M(t)=\dbE[M(T)]+\int_t^Tf^t(s)d\overleftarrow{B}(s)
=\dbE[M(T)]+\int_t^Tf(s)d\overleftarrow{B}(s)\qq for \ all \ t\in[0,T].$$
So the backward It\^{o} martingale representation theorem is obtained.
\end{proof}



\bibliographystyle{elsarticle-num}

\end{document}